\documentclass[12pt]{article} 
\usepackage{amsfonts,amssymb} 
\usepackage[cp1251]{inputenc}

\usepackage[margin=25mm, a4paper]{geometry}  

\def\blanksquare{\,\,\,$\sqcup\!\!\!\!\sqcap$}         
\def\qed{\hfill\blanksquare\linebreak\smallskip\par}   

\def\thname{Theorem}  \def\lmname{Lemma}    \def\prname{Proposition}
\def\dfname{Definition}  \def\crname{Corollary}  \def\rmname{Remark}
\def\exname{Example}   \def\conname{Conjecture}

\newtheorem{theorem}{\thname}[section]   
\newtheorem{lemma}{\lmname}[section]     
\newtheorem{corollary}[lemma]{\crname}   

\newtheorem{dftn}{\dfname}[section]
\newenvironment{definition}{\begin{dftn}\rm}{\end{dftn}} 

\newtheorem{rmrk}[lemma]{\rmname}

\def\bea#1{\begin{eqnarray*} #1 \end{eqnarray*}} \def\a{\!\!\!&\!\!\!\!&}

\def\url#1{#1}

\begin{document}

\title{Elementary solution to the fair division problem}
\author{Michael Blank and Maxim Polyakov\thanks{
        Institute for Information Transmission Problems RAS
        (Kharkevich Institute);}
        \thanks{National Research University ``Higher School of Economics'';} 
        \thanks{e-mail: blank@iitp.ru, pmaxol73@gmail.com}
       }
\date{May 2, 2024} 

\maketitle

\begin{abstract}
A new and relatively elementary approach is proposed for solving 
the problem of fair division of a continuous resource (measurable space, 
pie, etc.) between several participants, the selection criteria 
of which are described by charges (signed measures). 
The setting of the problem with charges is considered for 
the first time. The problem comes down to analyzing the properties 
of the trajectories of a specially constructed dynamical system 
acting in the space of finite measurable partitions. Exponentially 
fast convergence to a limit solution is proved for both the 
case of true measures and the case of charges.
\end{abstract}

\section{Introduction}
Problems associated with ``optimal'' division of resources under a given system 
of constraints are one of the most basic problems of mathematical economics.
Their fundamental complexity is determined by the large number (usually)
conflicting criteria for ``optimality'' of the division. The simplest example here
is the question of dividing a highly heterogeneous pie (treasure, apartment, 
etc.) between several participants, whose assessments of the parts they 
obtain may essentially differ from each other.

Thus, the fair division problem belongs to the class of multicriteria 
optimization problems; however, this problem has several important distinct 
features.
The main one is that instead of searching for an ``optimal'' (in one sense 
or another) solution (which may not exist), it is proposed to search for 
an acceptable solution. The latter setting in most cases turns out to 
be much simpler.
In a sense, this approach is a far generalization of the search for 
a base point (one of the vertices of a simplex) in the classical linear 
programming problem.

The problem of fair division in the simplest case is formulated as follows.
Suppose that we need to divide a pie among $r$ people so that everyone believes 
that he was not deprived. At the same time, the pie is highly heterogeneous, 
and each participant has his own criteria for comparing different pieces 
of the pie.
For example, someone likes a piece with strawberries better, while someone 
else notices that the left side of the pie is slightly burnt, etc.
The question arises of whether it is always possible to do this? 
Despite its seemingly “school-like” nature, this task turns out to be 
very difficult and a huge number of publications  was devoted to it in 
various settings 
(reviews of known results can be found, e.g., in \cite{BJK,Mo,Br,BBS}).

The first rigorous mathematical results for the case when the participants' 
criteria are determined by non-atomic countably additive measures,
were obtained in ~1961 by Dubins and Spanier \cite{DS}, who, relying on 
the results of Lyapunov \cite{La} and Halmos \cite{Ha} on the convexity 
property for vector measures, proved that for any set of criteria of the 
participants in the division (represented by non-atomic measures) 
the pie can be cut ``fairly''.
In 1980, W. Stromquist \cite{Str1} proved that this can be done in exactly 
$r-1$ cuts. However, this proof (as well as the result of \cite{DS})
is not constructive, and an explicit algorithm for such division is 
known only for $r<4$.
Much later, in 1999, F.E. Su \cite{Su} developed a fundamentally different 
combinatorial-topological approach to solve this problem,
based on an iterative exponentially fast converging algorithm.

Since we will be interested in approaching the fair division problem from 
the point of view of abstract measure theory, we do not consider numerous 
important and interesting geometric aspects of this problem such 
as representing a resource as a subset of Euclidean space, the possibility 
of partitioning by hyperplanes, or restricting partitions to connected 
sets only. The reader can find references to relevant results 
in the above reviews.

Until recently, it was believed that there was no final algorithm for 
solving the fair division problem (under some ``natural'' technical 
assumptions\footnote{The resource is represented by a unit segment, 
   and the partition elements are intervals.}) (see, for example, \cite{Str2}) .
However, in 2016, H.~Aziz and S.~McKenzie \cite{AM} proposed a finite 
algorithm. The algorithm is technically very complex, and the estimate 
$r^{r^{r^{r^{r^r}}}}$ for the number of steps required to construct 
a solution is astronomically large even for a small number of participants.

In real life, criteria for participants can naturally turn out 
to be much more complex and not reducible to classical measures. 
A complete study of functionals that allow the solution of the fair 
division problem has not been completed to date. The only progress 
here is related to the situation where it is not the ``best'' but 
the worst parts are chosen. Formally, in terms of measure theory,  
this means that the corresponding criterion is equal to some measure 
with a minus sign. A version of the finite Aziz-McKenzie algorithm for 
this case is proposed in \cite{DFHY}. 
Oddly enough, the combination of these options, i.e., the situation where  
the criteria are described by charges (signed measures), has not been 
studied in the mathematical literature.\footnote{As the reviewer 
pointed out to us, a particular example for 3 participants was considered 
in \cite{SH}.}
A possible explanation is that the authors were unable to reconcile 
in one algorithm the contradictions that arise in approaches with positive
and negative measures.

In this work, we propose a fairly simple exponentially fast convergent 
algorithm for solving the fair division problem for the case where  
criteria of participants are described by non-atomic countably 
additive charges.

Due to the non-atomicity of the functionals under consideration, 
various parts of the measurable space in a whole (for all participants 
simultaneously)  seem to be homogeneous.
Therefore, at first glance, our solution looks somewhat paradoxical.
The idea is to successively select disjunctive parts of the space on 
each of which it is possible to present an exact solution. 
The total remainder measure turns out to be exponentially small 
(relative to the step number of the division process),
which guarantees exponentially fast convergence of the process.

The paper is organized as follows. Section~2 provides a formal 
mathematical formulation of the problem; after that, facts 
on which the given algorithm is based are presented, and the 
main ideas of its construction are demonstrated in cases with 
a small number of participants. The last two sections describe 
the algorithm for constructing a strong solution for measures 
and its generalization to the case of charges.

\section{Formulation of the problem. Strong and weak solutions.}

Let $(M,\Sigma)$ be a measurable space, that is, a pair of a set $M$ 
and a $\sigma$-algebra of its measurable subsets $\Sigma$. Let us fix 
a natural number $r$, a set $\{\mu_i\}_{i=1}^{r}$ of functionals 
$\mu_i\colon \Sigma \to \mathbb{R}$,
and a measurable partition $\mathcal{F}:=\{F_i\}_{i=1}^{r}$ of the set $M$.

\begin{definition}
In the fair division problem, we call a partition $\mathcal{F}$
\begin{itemize}
\item\vskip-5pt
a \emph{\/ weak} solution, if 
    $\mu_i(F_i)\ge\frac{1}{r}\sum\limits_{j=1}^r\mu_i(F_j)$ for all $i$;
\item
a \emph{\/ strong}solution, if $\mu_i(F_i)\ge\mu_i(F_j)$ for all $i,j$.
\end{itemize}
\end{definition}

A weak solution to the problem corresponds to the case where the $i$-th 
participant has gained on average, relative to the functional $\mu_i$, 
no less than the others. A strong solution corresponds to the case where  
everyone believes that any other participant has gained no more than himself.

Note that in the literature on economics, a weak solution is often referred 
to as a ``proportional'', and a strong solution as an ``envy-free'' solution. 
In mathematics context, our terminology (first proposed by 
M.~Blank in \cite{Bl}) seems more convenient and adequate. Conceptually, 
this terminology goes back to the classical notions of weak and strong 
solutions in theory of differential equations.

\begin{lemma}
Let\/  $\mathcal{M}$ be a class of functionals for which the fair division 
problem has a solution. Then for any $r$ and any $\mu\in\mathcal{M}$
there is a partition of $M$ into $r$ measurable subsets 
$\left\{F_i\right\}_{i=1}^r$\textup   
such that \/ $\mu(F_i)=\mu(F_j)$, $\forall i,j$.
\end{lemma}

\begin{proof}
For the existence of any of the above types of solutions to the fair 
division problem, a necessary requirement is the presence of 
a corresponding solution for the case when all functionals $\mu_i$ 
coincide and are equal to ~$\mu$. In this case, the conditions for 
the weak and strong solutions are as follows:
\begin{itemize}
\item\vskip-5pt
$\mu(F_i)\ge\frac{1}{r}\sum\limits_{j=1}^r\mu(F_j)$, $\forall i$;
\item
$\mu(F_i)\ge\mu(F_j)$, $\forall i,j$.
\end{itemize}\vskip-5pt

Each of these conditions imply that $\mu(F_i)=\mu(F_j)$ for all $i,j$. 
Thus, there is a measurable partition of the set $M$ into $r$ parts $F_i$ 
that are equal with respect to the functional $\mu$.
\end{proof}

An important question is sufficiency of this condition for 
different types of the problem setting. A full answer to this 
question remains open, but the following theorem provides the most 
complete solution for today.

\begin{theorem}\label{t1}
For any non-atomic countably additive charges \/ 
$\left\{\psi_i\right\}_{i=1}^r$  there is
an explicit constructive algorithm for finding a strong solution to 
the fair division problem.
\end{theorem}

The proof of this result is an elementary construction presented below 
for constructing a strong solution for the case of non-atomic countably 
additive charges.

\section{Technical means }

The main statement that we will use when constructing a solution 
for charges is the Hahn-Jordan expansion of a charge $\psi$ on 
a set $M$ according to which there is a measurable partition 
$M=A^+ \cup A^-$ such that 
\begin{itemize}
\item\vskip-5pt
$\psi(a)\ge 0$ for any set $a\subset A^+$;
\item
$\psi(b)\le 0$ for any set $b\subset A^-$.
\end{itemize}\vskip-5pt
In other words, $\psi=\mu^+ - \mu^-$, where $\mu^\pm$ are countably 
additive measures that are restrictions of the charge $\psi$ on 
the sets $A^\pm$, respectively.

\begin{definition} 
We say that a charge $\psi$ is \emph{non-atomic} if its corresponding
measures $\mu^+$ and $\mu^-$ are non-atomic.
\end{definition}

Also, below we will need the following important statement, 
proved by A.A. Lyapunov for the case of vector measures (see \cite{La,Ha}).

\begin{theorem}
For any countably additive nonatomic measure \/ $\mu$ on \/ $(M,\Sigma)$, 
the set of values of the measure \/ $\mu$ on the $\sigma$-algebra \/ $\Sigma$ 
is convex.
\end{theorem}

\begin{corollary}\label{C1}
For any countably additive nonatomic measure \/ $\mu$ on \/ $(M,\Sigma)$\textup, 
any measurable set $A\subset M$ and any number \/ $\delta\in [0,\mu(A)]$ 
there is a measurable subset $B\subseteq A$ with \/ $\mu(B)=\delta$.
\end{corollary}

The last statement (sometimes referred to as W.~Sierpinski's theorem; 1922) 
states the existence of a set of any given measure.
In the framework of the approach under consideration, we assume that 
such a subset can be presented explicitly.

\section{Strong solution for the case of $r\in\{2,3\}$ measures}

\begin{definition}
For the partition $\mathcal{F}=\{F_i\}_{i=1}^{r}$, we will say
that the set $A$ is associated to the $j$th participant if $F_j=A$.
\end{definition}

Let us start with constructing a strong solution to the problem for 
non-atomic countably additive measures in the simplest situation of 
two participants in the division process, $r=2$ (the Cut-and-Choose algorithm).

By Corollary~\ref{C1}, $M$ has a subset $A_1$ of half the $\mu_1$-measure.
Then, for $A_2:=M\setminus A_1$, we have 
$$
\mu_1(A_2)=\mu_1(M\setminus A_1)=\mu_1(M)-\frac{1}{2}\mu_1(M)=\frac{1}{2}\mu_1(M).
$$
We obtain a partition of $M$ into two parts equal in $\mu_1$-measure.
Then if the 2nd participant is associated with a larger (or equal) part 
(denoted by $A_2$) with respect to the measure $\mu_2$, and the remaining 
part (denoted by $A_1$) is associated with the 1st, then we obtain 
a strong solution to our problem. Indeed,
$$\mu_1(A_1)=\mu_1(A_2) \Longrightarrow \mu_1(A_1)\geq\mu_1(A_2),$$
$$\mu_2(A_2)\geq \mu_2(A_1).$$

At first glance, it seems that this pattern of sequential ``halving'' 
can easily be extended to a larger number, say $r=3$, of participants. 
More precisely, the 1st participant divides $M$ into $r$ parts that 
are equal from his point of view, and the remaining participants 
order this partition, each in accordance with his own measure.
Next, all ``disputed'' elements of the partition are divided between 
pairs of participants claiming them. In fact, a large number of 
publications popularizing the problem of fair division reduce to 
this algorithm. The division algorithm is finite and extremely simple, 
but its detailed analysis shows that the result will only be a weak 
solution, since it may turn out that a participant who does not 
participate in any of the halvings will receive a share strictly less 
than one of other participants (see, for example, \cite{Str2}).

Below we will need the following technical statement.\footnote{Here 
   and  in what follows, $\bigsqcup$ denotes the union of 
   disjoint (non-intersecting) sets.}.

\begin{lemma}\label{l0}
Let \/ $\mu$ be a countably additive functional, 
$M=\bigsqcup_{i=1}^{\infty} H_i$\textup, 
and let on each set $H_i$ there be a strong solution \/  
$\mathcal{F}_i=\left\{F_{j,i}\right\}_{j=1}^n$.
Then the partition \/ $\mathcal{F}=\left\{F_j\right\}_{j=1}^n$ of $M$ 
with $F_j=\bigsqcup_{i=1}^{\infty} F_{j,i}$ is a strong solution on $M$.
\end{lemma}

\begin{proof}
First, $\mathcal{F}$ is indeed a partition of $M$:
$$
\bigsqcup\limits_{j=1}^n F_j=\bigsqcup\limits_{j=1}^n \bigsqcup\limits_{i=1}^\infty
F_{j,i}=\bigsqcup\limits_{i=1}^\infty \bigsqcup\limits_{j=1}^n
F_{j,i}=\bigsqcup\limits_{i=1}^\infty H_i=M.
$$
Second, since $\mu_j(F_{j,i})\ge\mu_j(F_{k,i})$, $\forall i,j,k$, we have  
$$
\mu_j(F_j)=\mu_j\biggl(\,\bigsqcup\limits_{i=1}^\infty
F_{j,i}\biggr)=\sum_{i=1}^\infty \mu_j(F_{j,i})\ge\sum_{i=1}^\infty
\mu_j(F_{k,i})=\mu_j\biggl(\,\bigsqcup\limits_{i=1}^\infty F_{k,i}\biggr)=\mu_j(F_k).
$$
This means $\mathcal{F}$ is a strong solution. \qed
\end{proof}

Using this result, we will analyze the construction of a strong solution for the case $r=3$.

Similar to the case $r=2$, we first divide $M$ into $r=3$ equal in measure $\mu_1$
parts $A_1$, $A_2$ and $A_3$. To do this, by Corollary~\ref{C1} we can choose 
a subset $A_1$ in $M$ such $\mu_1(A_1)=\frac{1}{3}\mu_1(M)$, and since
$$\mu_1(M\setminus A_1) = \mu_1(M)-\frac{1}{3}\mu_1(M) 
  = \frac{2}{3}\mu_1(M),$$
there is $A_2\subset M\setminus A_1$ such that
$\mu_1(A_2)=\frac{1}{3}\mu_1(M)$.

Letting $A_3:=M\setminus (A_1\cup A_2)$, we obtain 
\bea{\mu_1(A_3) \a= \mu_1(M\setminus (A_1\cup A_2))
                            = \mu_1(M\setminus A_1)-\mu_1(A_2) \\
                          \a=\frac{2}{3}\cdot\mu_1(M)-\frac{1}{3}\cdot\mu_1(M)
                             =\frac{1}{3}\cdot\mu_1(M) .} %

Let us order these sets with respect to the 2nd measure. 
Without loss of generality, we assume $\mu_2(A_1)\ge\mu_2(A_2)\ge\mu_2(A_3)$. 
Since $\mu_2(A_1)\ge\mu_2(A_2)$, there exists $A'\subset A_1$, such that 
$\mu_2(A')=\mu_2(A_2)$. Therefore, $\mu_1(A')\le\mu_1(A_1)=\mu_1(A_2)=\mu_1(A_3)$.

We want to divide the sets $A'$, $A_2$ and $A_3$ between the participants.
Let us associate with the third participant the largest of these sets 
relative to his measure $\mu_3$. 
If this is $A'$, then by associating $A_3$ with the first participant and 
$A_2$ with the second, we obtain a strong solution on $M':=A'\cup A_2\cup A_3$. 
Indeed, 
$$\mu_1(A')\leq\mu_1(A_2)=\mu_1(A_3),$$
$$\mu_2(A')=\mu_2(A_2)\geq\mu_2(A_3).$$

If $\mu_3(A_2)$ is the largest, then above the inequalities prove that 
if we associate $A_3$ with the first participant and $A'$ with the second, 
then the resulting partition turns out to be a strong solution.

Finally, if $\mu_3(A_3)$ is the largest, then to obtain a strong 
solution it is necessary to associate $A_2$ with the first participant 
and $A'$ with the second.

Next, we will look for a strong solution on the remaining part 
$M_1=M\setminus M'$.
Since $M'=A'\cup A_2\cup A_3$, then $M_1\subset A_1$, which means
$\mu_1(M_1)\leq \mu_1(A_1)=\frac{1}{3}\mu_1(M)$.
Note that the enumeration of measures was absolutely arbitrary, 
which means that our algorithm allows us to reduce any of the 
measures $\{\mu_i\}$ in the remaining part by a factor of more than 3. 
Thus, any of these remainder measures decreases exponentially. 
Therefore, in the limit, we get a partition of $M$ into 
a countable number of parts on which there is a strong solution, 
and the remainder is a set $M_\infty$ that is zero with respect 
to each measure (each partition $M_\infty$ is a strong solution).
Hence, by Lemma~\ref{l0} we obtain a strong solution on $M$.

\begin{definition}
A sequence of sets $\{A_n\}$ is called \emph{convergent} if the upper 
limit of this sequence  
$\bigcap\limits_{n=1}^\infty\bigcup\limits_{k=n}^\infty A_k$ 
coincides with the lower limit 
$\bigcup\limits_{n=1}^\infty\bigcap\limits_{k=n}^\infty A_k$.
\end{definition}

We will give a rigorous description of the construction of a sequence 
of sets converging to a strong solution when describing the algorithm 
for constructing a solution for the general case of $r\ge3$ participants.

The above elementary algorithm for constructing a strong solution for 
the case $r=3$ cannot be generalized to a larger number of participants. 
First of all, the point is that in the process of redistributing the parts 
$A_i$ among the participants, looping may occur. Therefore, retaining  
the inductive scheme of the division process, we propose another somewhat 
more complex design (see the next section).

It is also worth noting that even if we know how to solve the problem 
for measures, there still remains an important situation where functionals 
are charges. Let us discuss possible approaches to solving this problem. 
We start with the situation where the charges under consideration are 
strictly negative, that is, the $\{-\mu_i\}$ (taken with the opposite sign) 
are countably additive measures.

At first glance it seems that if these measures are 
probabilistic\footnote{A measure $\mu$ on $(M, \Sigma)$ is called 
   probabilistic if $\mu(M)=1$.}
then the setting under consideration is reduced to the previous one 
by replacing $\{\mu_i\}$ with a set of positive functionals $\{1-\mu_i\}$. 
However, the $1-\mu_i$ functionals are not additive, which means that they 
are not measures.

Indeed, if the functional $\phi=1-\mu$ were additive, then we would have 
\bea{\phi(A\sqcup B) \a= \phi(A)+\phi(B)=(1-\mu(A))+(1-\mu(B))=2-\mu(A)-\mu(B), \\
        \phi(A\sqcup B) \a= 1-\mu(A\sqcup B) = 1-\mu(A)-\mu(B).} %
We arrive at a contradiction:
$$
2-\mu(A)-\mu(B)=1-\mu(A)-\mu(B).
$$

Therefore, even for the case of strictly negative charges, a special
construction for the process of forming a strong solution is needed, 
which will be described in ~Section~\ref{S:6}.

\section{Construction of a strong solution for the case of $r>3$ measures.}

A special case of Theorem \ref{t1} is a situation where the functionals 
$\left\{\mu_i\right\}_{i=1}^r$ are measures. 
Below we will formulate and prove this statement.

\begin{theorem}\label{t2}
For any non-atomic countably additive measures 
\/ $\left\{\mu_i\right\}_{i=1}^r$ there is an explicit constructive 
algorithm for constructing a strong solution to the fair division problem.
\end{theorem}

\begin{definition}
For a disjunctive\footnote{Sets $\left\{A_i\right\}_{i=1}^r$ are called 
  disjunctive if they are pairwise disjoint.} 
set of sets $\left\{A_i\right\}_{i=1}^r$ by the {\em preferences of the $k$th 
participant}\footnote{As the reviewer pointed out to us, a similar 
     construction of preferences was described in the work \cite{SHHA}, 
     where some examples were analyzed for 3 and 4 participants.}
we call an index $j$ such that $\mu _k(A_j)\geq \mu _k(A_i)$ for any $i$.
\end{definition}

\begin{lemma}\label{l2}
For a given set of measures \/ $\left\{\mu_i\right\}_{i=1}^r$ and 
any index $k$ from an arbitrary measurable set $S\subset \Sigma$ we can 
select its subset $H\subset S$ of a ``large'' measure 
$$\mu_k(H)\geq \frac{1}{2^{r-1}} \mu_k(S) ,$$
on which there exists a strong solution to the fair division problem.
\end{lemma}

\begin{proof}
For convenience, we renumber the measures $\{\mu_i\}$ 
so that the measure $\mu_k$ has index $1$.

At the first step, we divide $S$ into $2^{r-1}$ parts 
$A_1^1,\ldots,A_{2^{r-1}}^1$ equal in measure $\mu_1$. 
To do this, we will sequentially consider subsets $A_i^1$ in $S$.
First, by Corollary~\ref{C1} there is a set $A_1^1$ such 
that $\mu_1(A_1^1)=\frac{1}{2^{r-1}}\mu_1( S)$.

Next, at the $t$-th step, in $S\setminus \bigcup_{i=1}^{t-1} A_i^1$ 
there is such a subset $A_{i}^1$ that 
$\mu_1(A_t^1)=\frac{1}{2^{r-1}}\mu_1(S)$, since
$$\mu_1(S\setminus \cup_{i=1}^{2^{r-1}-1} A_i^1) 
   = \frac{2^{r-1}-(2^{r-1}-1))}{2^{r-1}}\mu_1(S)
   = \frac{1}{2^{r-1}}\mu_1(S).$$ 

Let us denote the sets obtained at the $t$-th step by 
$A_1^t,\ldots,A_{2^{r-1}}^t$.
Each of our next steps (at which we will select some preferences
of the participants) will consist of changing a current collection of sets.
Let $K_i^t$ be the set of selected preferences of the $i$th participant 
at the $t$th step.
At the first step, $K_1^1=\{1,\ldots,2^{r-1}\}$ (all the $A_i^1$ are his 
preferences), while the remaining $K_i^1$ are set to be empty.

Next, at the $t$-th step, we do the following:

We arrange all sets $A_1^{t-1},\ldots,A_{2^{r-1}}^{t-1}$ according 
to measure $\mu_t$ and consider the $2^{r-t}$ largest among them:
$$\mu_1(A_{k_1}^{t-1})\geq \mu_1(A_{k_2}^{t-1})
 \geq\ldots\geq \mu_1(A_{k_{2^{r-t}}}^{t-1}).$$
Then in each set $A_{k_i}^{t-1}$ there is a subset
$$A_{k_i}'\subset A_{k_i}^{t-1}: \:
 \mu_t(A_{k_i}')=\mu_t(A_{k_{2^{r-t}}}^{t-1}).$$

The new collection $A_i^{t}$ is the collection $A_i^{t-1}$ in which elements 
$A_{k_i}^{t-1}$ are replaced by $A_{k_i}'$. The set of new selected 
preferences $K_i^{t}$ is such that $K_t^{t}=\{ k_1,\ldots,k_{2^{r-t}}\}$,
and $K_i^{t}=K_i^{t-1}\setminus K_t^{t}$ for $i\neq t$
(in particular, $K_i^t=\emptyset$ for $i>t$).

Let us verify by induction on $t$ that the following properties are 
satisfied at each step:
\begin{enumerate}
\item Sets $K_i^t$ are disjunctive.
\item $|K_j^{t}|\geq 2^{r-t}$ for $j\leq t$.
\item For any $j$ and any $k\in K_j^{t}$ it is true that $A_k^{t}$ 
      is the preference of the $j$th participant.
\end{enumerate}

Let us check the induction base:
\begin{enumerate}
\item $K_1^1=\{1,\ldots,2^{r-1}\}$, and the remaining $K_i^1$ are empty,
      which means they are disjunctive.
\item For $j=1$, we have $|K_j^{1}|=|K_1^{1}|=2^{r-1}$.
\item For $j=1$, for any $k$ it is true that $A_k^{1}$ is the 
      preference of the $1$th participant by the construction. 
      For $j\neq 1$, we have $K_j^1=\emptyset$.
\end{enumerate}

Proof of the induction step:
\begin{enumerate}
\item For $j>t$, the sets $K_j^t$ are empty, which means no set can intersect 
      with them. 
      For $s,k<t$, the sets $K_s^t=K_s^{t-1}\setminus K_t^t$ and 
      $K_k^t=K_k^{t-1}\setminus K_t^t$ are disjoint, 
      since $K_s^{t-1}$ and $K_k^{t-1}$ were disjoint by the 
      induction hypothesis.
      For $j=t$ and $s<t$, the sets $K_t^t$ and $K_s^t=K_s^{t-1}\setminus K_t^t$ 
      are disjoint. This means that $K_i^t$ are disjoint as well.

\item $|K_t^{t}|\geq 2^{r-t}$, since $|K_t^{t}|=2^{r-t}$, and for $j\leq t-1$
      $$|K_j^{t}|=|K_j^{t-1}\setminus K_t^t|\geq |K_j^{t-1}|-|K_t^{t}|
         \geq 2^{r-(t-1)}-2^{r-t} =2^{r-t}$$
      (by the induction hypothesis $|K_j^{t-1}|\geq 2^{r-(t-1)}$ for $j\leq t-1$).

\item For any $j\leq t-1$, we have $K_j^{t}=K_j^{t-1}\setminus K_t^t$.
      By the induction hypothesis, the $A_i^{t-1}$ with indices lying in $K_j^{t-1}$ 
      were the preferences of the $j$-th participant.
      The $t$th step consists in changing the sets with indices lying in $K_t^t$;  
      each of them is replaced by some subset and hence its measure 
      $\mu_j$ is reduced.
      Thus, the sets with indices from $K_j^{t}=K_j^{t-1}\setminus K_t^t$ 
      remain to be the preferences of the $j$th participant.
      For $k\in K_t^t$, we have $\mu_t(A_k^t)=\mu_t(A_{k_{2^{r-1}}}^{t-1})$; 
      hence, since $\mu_t(A_{k_{2^{r-1}}}^{t-1})$ is by construction the largest 
      of the values of $\mu_t(A_i^t)$, then $A_k^{t}$ is the preference of the 
      $j$th participant.
\end{enumerate}

As a result, at the $r$th step we have $|K_i^r|\geq2^{r-r}=1$ for any $i$;  
therefore, each participant has exactly one preference, and these preferences 
do not coincide with each other due to the disjunctiveness of the $\{K_i^r\}$.
Let us associate to the $i$th participant his preference $F_i$.

In this case, the first paeticipant is associated with a set $F_1$, such that
$\mu_{1}(F_1)=\frac{1}{2^{r-1}} \mu_1(S)$, since by the construction 
we remove from the selected preferences of the 1st participant those 
that we have already changed.

This means that the set $H=\sqcup_{i=1}^{r}F_i$ is the desired one, since
$\mu_1(H)\geq \mu_1(F_1)=\frac{1}{2^{r-1}} \mu_1(S)$, and also
$\mu_i(F_i)\geq \mu_i(F_j)$ for any $i$ and $j$, since $F_i$ is the 
preference of the $i$th participant. \qed
\end{proof}

\begin{definition} The set $H$ obtained in Lemma~\ref{l2} will be called 
the \emph{$\mu_k$-satisfying subset} of $S$.\end{definition}

Let us construct a new sequence of sets $\{H_i\}$ according to the 
following rules:
\begin{itemize}
\item $H_1$ is a $\mu_1$-satisfying subset $M=M_0$.
\item If $k$ mod $r\ne 0$, then $t=$ $k$ mod $r$; otherwise $t=r$
\item $H_k$ is the $\mu_t$-satisfying subset of $M\setminus \cup_{i=1}^{k-1} H_i$.
\item $M_s=M\setminus \cup_{i=1}^{sr} H_i$.
\end{itemize}

Define also the averaged measure $\mu(A):=\frac{1}{r}\sum_{i=1}^r\mu_i(A)$.

In other words, for each measure we in turn select sets that are not less 
than $1/2^{r-1}$ of the current measure of the remainder 
$M\setminus \cup_{i=1}^{k-1} H_i$.
Indeed,
$$
\mu_t(H_k)\ge 2^{-(r-1)}\mu_t\Biggl(M\setminus\bigcup_{i=1}^{k-1} H_i\Biggr).
$$

\begin{lemma}
We have $\mu(M_s)\leq \frac{2^{r-1}-1}{2^{r-1}} \mu(M_{s-1})$.
\end{lemma}

\begin{proof}
For any $j$ we have
\bea{\mu_j(\cup_{i=(s-1)r+1}^{sr} H_i) 
             \a= \mu_j(H_{(s-1)r+j}) 
               + \mu_j(\cup_{i=(s-1)r+1}^{sr} H_i\setminus H_{(s-1)r+j}) \\ 
             \a\geq 2^{-(r-1)}\cdot \mu_j(M\setminus \cup_{i=1}^{(s-1)r+j-1} H_i) 
               + \mu_j(\cup_{i=(s-1)r+1}^{sr} H_i\setminus H_{(s-1)r+j}) \\
             \a\geq 2^{-(r-1)}\cdot (\mu_j(M_{s-1}\setminus \cup_{i=(s-1)r+1}^{(s-1)r+j-1} H_i) 
               + \mu_j(\cup_{i=(s-1)r+1}^{sr} H_i\setminus H_{(s-1)r+j})) \\
             \a\geq 2^{-(r-1)}\cdot \mu_j((M_{s-1}\setminus \cup_{i=(s-1)r+1}^{(s-1)r+j-1} H_i) 
                                         \cup (\cup_{i=(s-1)r+1}^{sr} H_i\setminus H_{(s-1)r+j})) \\
             \a= 2^{-(r-1)}\cdot \mu_j(M_{s-1}\cup (\cup_{i=(s-1)r+j+1}^{sr} H_i))) \\
             \a\geq 2^{-(r-1)}\cdot\mu_j(M_{s-1}). } %
Hence, %
\bea{\mu_j(M_s) \a= \mu_j(M_{s-1}\setminus\cup_{i=(s-1)r+1}^{sr} H_i) \\
                          \a= \mu_j(M_{s-1})-\mu_j(\cup_{i=(s-1)r+1}^{sr} H_i) \\
                          \a\leq \frac{2^{r-1}-1}{2^{r-1}}\cdot \mu_j(M_{s-1}) .}%
Thus,
$$ \sum_{j=1}^r\mu_j(M_{s})
 \leq \frac{2^{r-1}-1}{2^{r-1}} \sum_{j=1}^r\mu_j(M_{s-1}) .$$
Finally we obtain 
$\mu(M_{s})\leq\frac{2^{r-1}-1}{2^{r-1}} \mu(M_{s-1})$. \qed
\end{proof}

\begin{corollary}\label{c-1} 
We have $\mu(M_s)\leq (\frac{2^{r-1}-1}{2^{r-1}})^{s} \mu(M)$.
\end{corollary} 

\begin{corollary} Let $M_{\infty}=M\setminus \cup_{i=1}^{\infty}H_i$, 
       then $\mu_j(M_{\infty})=0$ for every $j$.
\end{corollary} 

\begin{proof}
Let us tend the parameter $s$ in Corollary~\ref{c-1} to infinity.
Since $\frac{2^{r-1}-1}{2^{r-1}}<1$, we have $\mu(M_{\infty})\leq 0$,
and from the non-negativity of the measure $\mu$ it follows that 
$\mu(M_{\infty})=0$.
Since $\mu(M_{\infty})=\Sigma_{j=1}^n\mu_j(M_{\infty})$, 
we obtain the desired statement. \qed
\end{proof}

Thus, we represented the set $M$ as follows:
$M=\bigsqcup_{i=0}^{\infty} H_i$,
where on each $H_i$ there is a strong solution 
$\mathcal{F}_i=\{F_{j,i}\}_{j=1}^{n}$.
(On $M_{\infty}$ there is a strong solution for which the 1st measure 
is associated with $M_{\infty}$,
and the others with an empty set, so we will assume that $M_{\infty}=H_0$).
This means that by Lemma \ref{l0} there exists a strong solution on $M$.

For each measure $\mu_j$, consider the sequence of sets 
$N_{j,k}=\bigsqcup_{i=0}^{kr} F_{j,i}$.

\begin{lemma}
The sequence of sets $N_{j,k}$ converges as $k\to\infty$, and the limit of 
$N_{j,k}$ is equal to $\bigcup_{k=1}^{\infty} N_{j,k} =: N_j^{\infty}$.
\end{lemma}

\begin{proof}
Since $N_{j,k}\subset N_{j,k+1}$, we have 
$$\bigcup_{k=n}^{\infty} N_{j,k}=\bigcup_{k=1}^{\infty} N_{j,k}= N_j^{\infty}.$$
This means that the upper limit of $N_{j,k}$ satisfy the relation
$$\bigcap_{n=1}^{\infty}\bigcup_{k=n}^{\infty} N_{j,k}
=\bigcap_{n=1}^{\infty} N_j^{\infty}= N_j^{\infty}.$$
On the other hand,
$\bigcap_{k=n}^{\infty} N_{j,k}=N_{j,n}$. 
Therefore, the lower limit is
$$\bigcup_{n=1}^{\infty}\bigcap_{k=n}^{\infty} N_{j,k}
=\bigcup_{n=1}^{\infty} N_{j,n}= N_j^{\infty},$$
\end{proof}

\begin{lemma}
$\{N_j^{\infty}\}_{j=1}^{r}$ is a strong solution on $M$.
\end{lemma}

\begin{proof}
Since the sets $F_{j,i}$ do not intersect for different $j$, then
$N_j^{\infty}\bigcap N_l^{\infty}=\emptyset$ for $j\neq l$.

On the other hand,
$$\bigcup_{j=1}^{r} N_j^{\infty}
 = \bigcup_{j=1}^{r} \bigcup_{k=1}^{\infty} N_{j,k}
 = \bigcup_{j=1}^{r} \bigcup_{k=1}^{\infty} \bigsqcup_{i=0}^{kr} F_{j,i}
 = \bigcup_{j=1}^{r} \bigcup_{i=0}^{\infty} F_{j,i}
 = \bigcup_{i=0}^{\infty} H_i = M.$$

Therefore $M=\bigsqcup_{j=1}^{r} N_j^{\infty}$.

Since by Lemma~\ref{l0} we have $N_j^{\infty}=\bigcup_{k=1}^{\infty} N_{j,k}$,  
we conclude that $\mu_j(N_j^{\infty})\geq \mu_j(N_l^{\infty})$ for all $j,l$, 
which completes the proof. \qed
\end{proof}

\section{Construction of a strong solution for charges}\label{S:6}

Now we pass to the case of non-atomic countably additive charges 
$\{\psi_i\}_{i=1}^r$.

For each charge $\psi_i$, consider the corresponding sets $A^+_i$ and $A^-_i$
from the Hahn-Jordan expansion. Then intersections of sets of sets 
$\{A^+_i,A^-_i\}$ split the entire set $M$ into $2^r$ parts, the restriction  
of $\psi_i$ to each of these parts is either a measure or a minus measure.

Therefore, according to Lemma \ref{l0}, it suffices to find out how 
to solve the problem on each of these $2^r$ parts. According to the 
previous section, the problem has already been solved for those parts 
where all restrictions of the charges are positive.

Call the resulting parts by $X_1, \ldots, X_{2^r}$ and consider 
one of these parts, which we will denote for simplicity by $X$.

\begin{definition}
Denote by $I^+$ ($I^-$) the set of indices $i$ for which the charge restriction
$\psi_i$ to $X$ is a measure (minus measure). \end{definition}

There are 2 possible cases: $I^+\ne \emptyset$ and $I^+= \emptyset$. 
Let us first consider the first of them.

\begin{lemma}
If $I^+\ne \emptyset$, then $X$ has a strong solution.
\end{lemma}

\begin{proof}
By Theorem \ref{t2}, on the set $X$ there is a strong solution 
$\mathcal{F^+}:=\{F_i\}_{i\in I^+}$ for measures $\{\psi_i\}_{i\in I^+}$.
Let us extend the collection $\mathcal{F^+}$ to a collection 
$\mathcal{F}:=\{F_i\}_{i=1}^{r}$ by the 
sets $F_i=\emptyset$ for $i\in I^-$.

Then for $i\in I^-$ and any $j$ we have 
$\psi_i(F_i)=\psi_i(\emptyset)=0\geq \psi_i(F_j)$,
since the $i$-th signed measure is negative on the considered set.

If $i\in I^+$ and $j\in I^+$, then $\psi_i(F_i)\geq \psi_i(F_j)$,
since $\mathcal{F^+}$ is a strong solution for positive measures.
If $i\in I^+$ and $j\in I^-$, then
$\psi_i(F_i)\geq 0 = \psi_i(\emptyset)=\psi_i(F_j)$. \qed
\end{proof}

It remains to construct a strong solution on such a set that the restriction 
of any considered charge to which is a minus measure. To do this, 
we notice that the problem of finding a strong solution for 
minus measures is equivalent to finding a partition as follows.

\begin{definition}
In the fair division problem with a set of measures $\{\mu_i\}_{i=1}^r$, 
we call a partition $\mathcal{F}$ a \emph{gentleman's} solution if 
$\mu_i(F_i)\leq \ mu_i(F_j)$ for any $i,j$.
\end{definition}

In the literature on economics, this setting is often called ``chore division''.
It is not difficult to understand that from the point of view of solving the 
problem of fair division with a set of strictly negative charges 
$\{\psi_i=-\mu_i\}_{i=1}^r$ these two settings are equivalent.

Similar to Lemma \ref{l0}, the following its variation can be formulated an proved.

\begin{lemma} \label{l6}
Let $\mu$ be a countably additive functional, $M=\sqcup_{i=1}^{\infty} H_i$ 
and assume that on each set $H_i$ there is a gentleman's solution
$\mathcal{F}_i=\{F_{j,i}\}_{j=1}^n$. Then the partition
$\mathcal{F}=\{F_{j}\}_{j=1}^n$ with $F_j=\sqcup_{i=1}^{\infty} F_{j,i}$ 
is a strong solution on $M$.
\end{lemma}

\begin{proof}
Since $\mu_j(F_{j,i})\leq \mu_j(F_{k,i})$ for all $i,j,k$, then %
$$
\mu_j(F_j)=\mu_j\biggl(\,\bigsqcup\limits_{i=1}^\infty F_{j,i}\biggr)
=\sum_{i=1}^\infty \mu_j(F_{j,i})\le\sum_{i=1}^\infty \mu_j(F_{k,i})=
\mu_j\biggl(\,\bigsqcup\limits_{i=1}^\infty F_{k,i}\biggr)=\mu_j(F_k).
$$\qed
\end{proof}

Let a set $S\in\Sigma$ satisfy the condition $\prod_{i=1}^r\mu_i(S)\ne0$.

\begin{lemma} \label{l4}
There are sets $R^1,\ldots,R^r$, $\sqcup_i R^i\subset S$ and a renumbering 
of the measures $\mu_{l_1},\ldots,\mu_{l_r}$ such that for each $k$ we have
$\mu_{l_k}(R^{k})=\frac{1}{r}\mu_{l_k}(S)$, and for $j>k$ we have 
$\mu_{l_j}(R^{ k})\leq \frac{1}{r}\mu_{l_k}(S)$. 
\end{lemma}

\begin{proof}
We will successively construct this collection of sets. At the 1st step, 
we consider $R^1_1\subset S$ such that 
$\mu_1(R^1_1)=\frac{1}{r} \mu_1(S)$. 
Further, if at the $t$th step 
$\mu_t(R_{t-1}^1)>\frac{1}{r} \mu_t(S)$, then 
$R_t^1$ is such subset $R_{t-1}^1$ 
that $\mu_t(R_{t}^1)=\frac{1}{r}\mu_t(S)$; otherwise, $R_t^1=R_{t-1}^1$.

At the end of the $r$-th step we obtain the set $R^1=R_r^1$. Let $l_1$ 
denote the step number after which we no longer changed $R_i^1$. 
Then, by the construction, 
$\mu_{l_1}$ and $R^1$ satisfy the following conditions:
$$ \mu_{l_1}(R^1)=\frac{1}{r} \mu_{l_1}(S), \qquad
 \mu_{j}(R^1)\leq \frac{1}{r} \mu_{j}(S)~~j\neq l_1 .$$.

Now at the $k$th step, in a similar way we will select $R^k$ and 
a measure $\mu_{l_k}$ from the set $S\setminus \sqcup_{i=1}^{k-1}R^i$ 
for all measures except $\mu_{l_1},\ldots,\mu_{l_{k-1}}$,
such that $\mu_{l_k}(R^k)=\frac{1}{r} \mu_{l_k}(S)$,
and $\mu_{j}(R^k)\leq \frac{1}{r} \mu_{j}(S)$ for all 
$j\notin\{l_1,\ldots,l_k\}$.
We only need to check that we can do the 1st stage of the $k$-th step, 
that is, with respect to some measure $\mu_j$ for $j\notin\{l_1,\ldots,l_{k-1}\}$ 
remainder measure is not less than $\frac{1}{r}\mu_{j}(S)$.

\bigskip

By the construction, $\mu_j(R^i)\leq \frac{1}{r}\mu_j(S)$ for any 
$j\notin\{l_1,\ldots,l_{k-1}\}$.
Hence,
$$\mu_j(S\setminus \bigsqcup_{i=1}^{k-1}R^i) = \mu_j(S)-\sum_{i=1}^{k-1}\mu_j(R^i)
 \geq \mu_j(S)\left(1-\frac{k-1}{r}\right).$$
Therefore, for $k\leq r$ we can take this step, and as a result we obtain 
the desired ordering of the measures $\mu_{l_i}$ and the set $R^i$. \qed
\end{proof}

\begin{definition}
In the case of a gentleman's solution for a disjunctive collection of sets 
$\{A_i\}_{i=1}^r$ the 
\textit{preference of the $k$th participant} is an index $j$ such that
$\mu _k(A_j)\leq \mu _k(A_i) ~~\forall i$. \end{definition}

\begin{lemma} \label{l3}
For a given set of measures $\{\mu_i\}_{i=1}^r$ from any set $S\subset M$ 
one can select a subset $H$ such that if $\{\mu_{l_i}\}_{i=1}^r$ is the 
ordering of measures for $S$ from Lemma \ref{l4}, then
$$\mu_{l_1}(H)\geq \frac{1}{2^{r-1}} \mu_{l_1}(S),$$
and $H$ has a gentleman's solution.
\end{lemma}

\begin{proof}
To simplify the notation, we will enumerate the measures $\mu_i$, 
rather than $\mu_{l_i}$, 
remembering that the enumeration is not arbitrary, since it depends on the 
set that we are dividing.

First, we divide $R^1$ into $2^{r-1}$ parts equal in measure to $\mu_{1}$. 
Let us denote the resulting partition elements at the $t$-th step by 
$A_1^t,\ldots,A_{2^{r-1}}^t$, and denote the selected preferences of the $i$-th 
participant at the $t$-th step sets by $K_i^t$. 
At the first step we set $K_1^1:=\{1,\ldots,2^{r-1}\}$
(all $A_i^1$ are his preferences), and leave the remaining $K_i^1$ empty.

Next, at the $t$-th step, we do the following.

Order all sets $A_1^{t-1},\ldots,A_{2^{r-1}}^{t-1}$ according to 
the measure $\mu_t$ and consider $2^{r-t}$ smallest sets among them:
$$\mu_1(A_{k_1}^{t-1})\leq \mu_1(A_{k_2}^{t-1})\leq\ldots 
  \leq \mu_1(A_{k_{2^{r-t} }}^{t-1}).$$

We want to separate from $R^t$ such a collection of sets 
$C_1^{t-1},\ldots,C_{2^{r-t}}^{t-1}$ 
that $\forall i$
$$\mu_t(C_i^{t-1})=\mu_t(A_{k_{2^{r-t}}}^{t-1})-\mu_t(A_{k_i}^{t-1}) .$$

Then the new set $A_i^{t}$ is the ``adjusted'' set $A_i^{t-1}$ 
in which the elements $A_{k_i}^{t-1}$ are replaced with  
$A_{k_i}^{t-1}\bigsqcup C_i^{t-1}$.
The set of new selected preferences $K_i^{t}$ is such that
$K_t^{t}=\{ k_1,\ldots,k_{2^{r+1-t}}\}$ and 
$K_i^{t}=K_i^{ t-1}\setminus K_t^{t}$ for $i\neq t$
(in particular $K_i^t=\emptyset$ for $i>t$).

Note that we can make this step only if
$$\sum_{i=1}^{2^{r-t}} \mu_t(C_i^{t-1})\leq \mu_t(R^t).$$
Further, %
\bea{\sum_{i=1}^{2^{r-t}} \mu_t(C_i^{t-1}) 
        \a= \sum_{i=1}^{2^{r-t}} (\mu_t(A_{k_{2^{r-t}}}^{t-1})-\mu_t(A_{k_i}^{t-1}))
              \leq 2^{r-t}\cdot \mu_t(A_{k_{2^{r-t}}}^{t-1}) \\
        \a\leq \frac{2^{r-t}}{2^{r-1}-2^{r-t}}\cdot\sum_{i=2^{r-t}+1}^{2^{r-1}} \mu_t(A_{k_i}^{t-1}).}
Since all sets $A_i^{t-1}$ are contained in the union $R^1, \ldots,R^{t-1}$, then
$$
\mu_t\Biggl(\,\bigsqcup\limits_{i=1}^{2^{r-1}}
A_i^{t-1}\Biggr)\le\sum_{j=1}^{t-1}\mu_t(R^j) \le\frac{t-1}{r}\mu_t(S).
$$
The last inequality follows from the construction of $R^i$.

Thus,
$$
\sum_{i=2^{r-t}+1}^{2^{r-1}} \mu_t\bigl(A_{k_i}^{t-1}\bigr)
\le\mu_t\Biggl(\,\bigsqcup\limits_{i=1}^{2^{r-1}}
A_i^{t-1}\Biggr)\le\frac{t-1}{r}\mu_t(S).
$$
Since $\frac{t-1}{2^{t-1}-1}\leq 1$ for $t\geq 2$, then %
\bea{\frac{2^{r-t}}{2^{r-1}-2^{r-t}} \sum_{i=2^{r-t}+1}^{2^{r-1}} \mu_t(A_{k_i}^{t-1})
   \a\leq\frac{2^{r-t}}{2^{r-1}-2^{r-t}} \cdot \frac{t-1}{r}\cdot\mu_t(S) \\
   \a= \frac{t-1}{(2^{t-1}-1) r}\cdot\mu_t(S)
      \leq \frac{1}{r}\mu_t(S)=\mu_t(R^t) .}

Now, similarly to the previous considered situation, we prove 
by induction on $t$ the following statements:
\begin{enumerate}
\item The sets $K_i^t$ are disjunctive.

\item $|K_j^{t}|\geq 2^{r-t}$ for $j\leq t$.

And to prove the 3rd point we need the following reasoning.

\item For any $j$ and $k\in K_j^{t}$ it is true that $A_k^{t}$ is a preference 
for the $j$-th measure.
It is necessary to note that changing the parts is made by combining them,
which means that the parts do not decrease with respect to any measure and 
therefore cannot become preferences if they were not originally.
Moreover, by the construction, the set $K_t^t$ consists only of preferences of 
the $t$th measure.
\end{enumerate}

At the $r$th step, each participant will have at least $2^{r-r}=1$ different 
preferences. Let us associate each participant with his preference $F_i$. 
The first participant is associated with a set $F_1$ such that
$\mu_{1}(F_1)=1/2^{r-1} \mu_1(R^1)$, since by the construction we remove 
from the selected preferences of the 1st participant those that we have changed.

It should also be noted that $\mu_{1}(F_{1})\leq \mu_{1}(F_j)$ for any $j$,
since the original sets $A_i^1$ can only be enlarged. Therefore,  %
\bea{\mu_{1}(H) \a= \mu_{1}(\sqcup_{j=1}^{r}F_j)
                             =\sum_{j=1}^{r}\mu_{1}(F_j) \geq \sum_{j=1}^{r}\mu_{1}(F_{1}) \\
                          \a= r \cdot \mu_{1}(F_{1})= r \cdot 1/2^{r-1}\cdot \mu_{1}(R^1) \\
                          \a= r \cdot 1/2^{r-1}\cdot \frac{1}{r} \cdot \mu_{1}(S) 
                             = 2^{-(r-1)}\cdot \mu_{1}(S). }
Thus, the partition $H=\bigsqcup_{i=1}^{r}F_i$ satisfies all the required 
conditions. \qed
\end{proof}

\begin{theorem}
For non-atomic countably additive measures $\{\mu_i\}_{i=1}^r$ there 
exists a gentleman's division.
\end{theorem}

\begin{proof}
Let us prove by induction on the number of participants that there is 
a gentleman's division of $M$.

\textbf{Let us check the induction base $r=1$.}
For the 1st participant there is a single partition with $F_1=M$ and 
it is obviously gentlemanly.

\textbf{Let us check the induction step.}

\begin{definition} We call the set $H$ obtained in Lemma~\ref{l3} the 
$\mu_{l_1}$-satisfying subset of $S$.\end{definition}

Let us define by induction the following sequence of sets $H_i$:
\begin{itemize}
\item $H_1$ is the $\mu_{h_1}$-satisfying subset of $M=M_0$;
\item $H_k$ is the $\mu_{h_k}$-satisfying subset of 
      $M\setminus \bigcup_{i=1}^{k-1} H_i$; 
\item $M_{\infty}=M\setminus \bigsqcup_{i=1}^{\infty}H_i$.
\end{itemize}

All indices $h_1,h_2,\ldots$ belong to the finite set $\{1,\ldots,r\}$.
This means that one of them is repeated infinitely many times; 
denote it by $\ell$.
\end{proof}

\begin{lemma}\label{l5}
Let $k_1< k_2< \ldots$ be the set of all indices $i$ for which $h_i=\ell$. Then
$$
\mu_{\ell}\Biggl(M\setminus\bigsqcup\limits_{i=1}^{k_{i+1}} H_i\Biggr)
\le\left(\frac{2^{r-1}-1}{2^{r-1}}\right)^{i+1}\mu_{\ell}(M).
$$
\end{lemma}

\begin{proof} We have 
\begin{proof}
\bea{\mu_{\ell}(\sqcup_{i=k_{i}+1}^{k_{i+1}} H_i) 
       \a= \mu_{\ell}(H_{k_{i+1}}) + \mu_{\ell}(\cup_{i=k_i+1}^{k_{i+1}} H_i\setminus H_{k_{i+1}}) \\
       \a\geq 2^{-(r-1)}\cdot \mu_{\ell}(M\setminus \cup_{i=1}^{{k_{i+1}}-1} H_i) 
           + \mu_{\ell}(\cup_{i=k_i+1}^{k_{i+1}} H_i\setminus H_{k_{i+1}}) \\
       \a\geq 2^{-(r-1)}\cdot( \mu_{\ell}(M\setminus \cup_{i=1}^{{k_{i+1}}-1} H_i) 
           + \mu_{\ell}(\cup_{i=k_i+1}^{k_{i+1}} H_i\setminus H_{k_{i+1}})) \\
       \a\geq 2^{-(r-1)}\cdot\mu_{\ell}(M\setminus \sqcup_{1}^{k_{i}} H_i) .}%
Hence,
$$
\mu_{\ell}\Biggl(M\setminus\bigsqcup\limits_{i=1}^{k_{i+1}} H_i\Biggr)
\le\frac{2^{r-1}-1}{2^{r-1}}\mu_{\ell}\Biggl(M\setminus\bigsqcup\limits_{i=1}^{k_{i}}
H_i\Biggr).
$$
Finally we obtain 
$$
\mu_{\ell}\Biggl(M\setminus\bigsqcup\limits_{i=1}^{k_{i+1}} H_i\Biggr)
\le\left(\frac{2^{r-1}-1}{2^{r-1}}\right)^{i+ 1}\mu_{\ell}(M).
$$\qed
\end{proof}

\begin{corollary}
Letting the number of indices $i$ considered in Lemma~\ref{l5} tend to infinity, 
we obtain $\mu_{\ell}(M_\infty)=0$.
\end{corollary}

By the induction hypothesis, for measures with indices 
$i\in \{1,2\ldots r\}\setminus {\ell}$
there is a gentleman's solution $\mathcal{F}^{*}=\{F_i\}$ on $M_{\infty}$.
Then, for $F_{\ell}=\emptyset$, the partition 
$\mathcal{F}:=\{F_i\}_{i=1}^{r}$ of  $M_{\infty}$
is a gentleman's solution for all $r$ participants.

Indeed, since $\mathcal{F}^{*}$ is a gentleman's solution for measures
with indices $i\in \{1,2\ldots r\}\setminus {\ell}$, then we only need to 
check the inequalities containing the conditions associated with the 
$\ell$th participant:
$$\mu_i(F_i)\geq 0=\mu_i(\emptyset)=\mu_i(F_\ell),$$
$$\mu_\ell(F_\ell)=\mu_\ell(\emptyset)=0=\mu_\ell(M_{\infty})\geq \mu_\ell(F_i) .$$

We have obtained the representation 
$M=M_\infty\sqcup \Bigl(\bigsqcup\limits_{i=1}^\infty H_i\Bigr)$, 
and at each element of the partition there is a gentleman's solution. 
This means that the required partition exists by  Lemma~\ref{l6}. \qed
\end{proof}

At the request of an anonymous reviewer to present a simple, understandable 
example, in low dimension, for two participants, we will analyze this 
simplest situation in detail.
Note, firstly, that neither the dimension, nor geometry in the setting 
under consideration have any significance, 
and secondly, for two participants, a strong solution for charges can be 
obtained by the Cut-and-Choose method. However, perhaps a significantly 
more complex construction of the formal implementation of the algorithm described 
in Section~\ref{S:6} for $r=2$ will help to understand the general situation 
for $r>2$.

Given charges $\psi_1,\psi_2$, we construct two partitions of the space
$$M=M_1^+\sqcup M_1^-=M_2^+\sqcup M_2^-,$$
corresponding to positive and negative parts of the charges $\psi_1,\psi_2$.
We set
$$ M^{++}:=M_1^+ \cap M_2^+, ~ M^{+-}:=M_1^+ \cap M_2^-, 
   ~M^{-+}:=M_1^- \cap M_2^+, ~ M^{--}:=M_1^- \cap M_2^- .$$
Note that $\{M^{++},M^{+-},M^{-+},M^{--}\}$ is again a partition of $M$.
According to the construction of a strong solution for charges, the parts 
$M^{++}$ and $M^{--}$ are divided between the participants ``in half'', 
while the part $M^{+-}$ is given entirely to the 1st participant, and 
the part $M^{-+}$ is given entirely to the 2nd participant.

This completes the process of constructing the strong solution for 2 participants, 
since the need for an iterative scheme arises only for $r>2$.

\section{Conclusion}

In this paper, we have proposed a simple exponentially converging 
algorithm for constructing a strong solution to the fair division 
problem for the case of individual criteria described by non-atomic charges.
As was already noted in the Introduction, there are publications \cite{AM,DFHY} 
that describe finite algorithms for solving this problem for the cases
of strictly positive and strictly negative measures, respectively
(albeit with an astronomical estimate of the number of necessary operations).
Unfortunately, these publications are not written at a completely mathematical 
level of rigor (the proofs given in them are rather detailed descriptions 
of the program flow charts), and the algorithms themselves are so complex 
that checking the adequacy of their operation seems difficult.
Apparently, something like a computer-assisted roof is needed here.
In any case, we were unable to check them ``manually'' even for a small number 
of participants. However, if these algorithms really work, then their 
combination with the scheme described in Section~\ref{S:6} for dividing 
the space into positive and negative parts of the charges and working 
with their combinations lead to a finite algorithm for constructing 
a strong solution.

Note that despite the seeming similarity of the fair division problem under 
consideration with the better known setting of multicriteria Pareto 
optimization, they are fundamentally different. Pareto optimality is understood 
as a state of the system in which not a single criteria can be improved 
without deteriorating any other criteria. One can easily see that any 
Pareto optimal state is a strong solution to the fair division problem, 
but the latter setting is much more flexible and allows 
to explore a wider class of systems.


\end{document}